\documentclass[11pt]{amsart}
\usepackage{amssymb,amsmath,amsthm,enumerate}

\theoremstyle{plain}
\newtheorem{theorem}{\bf Theorem}[section]
\newtheorem*{theorem*}{Theorem 1.1$'$}
\newtheorem{lemma}[theorem]{\bf Lemma}
\newtheorem{proposition}[theorem]{\bf Proposition}

\numberwithin{equation}{section}

\DeclareMathOperator{\sign}{sign}

\renewcommand{\Re}{\operatorname{Re}}

\newcommand{\abs}[1]{\lvert#1\rvert}
\newcommand{\norm}[1]{\lVert#1\rVert}
\newcommand{\jap}[1]{\langle#1\rangle}

\newcommand{\diag}{\mathrm{diag}}

\newcommand{\bbR}{{\mathbb R}}
\newcommand{\bbC}{{\mathbb C}}
\newcommand{\bbN}{{\mathbb N}}
\newcommand{\calO}{\mathcal{O}}

\begin{document}

\title[Spectral asymptotics for arithmetical matrices]{Spectral asymptotics for a family of arithmetical matrices and connection to Beurling primes}

\author{Titus Hilberdink}
\address{Nanjing University of Information Science and Technology (Reading Academy), 219 Ningliu Road, Pukou District, Nanjing, Jiangsu, 210044 China}\email{t.w.hilberdink@nuist.edu.cn, t.hilberdink@reading.ac.uk}

\author{Alexander Pushnitski}
\address{Department of Mathematics, King's College London, Strand, London, WC2R~2LS, U.K.}
\email{alexander.pushnitski@kcl.ac.uk}


\dedicatory{To Fritz Gesztesy on the occasion of his 70th birthday with warmest wishes}



\begin{abstract}
We consider the family of arithmetical matrices given explicitly by 
$$
E=\left\{\frac{[n,m]^t}{(nm)^{(\rho+t)/2}}\right\}_{n,m=1}^\infty
$$
where $[n,m]$ is the least common multiple of $n$ and $m$ and the real parameters $\rho$ and $t$ satisfy 
$t>0$, $\rho>t+1$. We prove that $E$ is a compact self-adjoint operator on $\ell^2(\bbN)$ with infinitely many of both positive and negative eigenvalues.
Furthermore, we prove that the ordered sequence of positive eigenvalues of $E$ obeys the asymptotic relation
$$
\lambda^+_n(E)=\frac{\varkappa}{n^{\rho-t}}(1+o(1)), \quad n\to\infty,
$$
with some $\varkappa>0$ and the negative eigenvalues obey the same relation, with the same asymptotic coefficient $\varkappa$. 
We also indicate a connection of the spectral analysis of $E$ to the theory of Beurling primes. 
\end{abstract}

\maketitle

\section{Introduction and main result}

\subsection{Introduction}
For natural numbers $n$ and $m$, we denote by $(n,m)$ the greatest common divisor (GCD) of $n$ and $m$ and by $[n,m]$ the least common multiple (LCM) of $n$ and $m$. In 1875, Smith \cite{Smith} computed the determinant of the $N\times N$ matrix 
\[
\{(n,m)\}_{n,m=1}^N
\]
(the answer is the product $\phi(1)\phi(2)\cdots\phi(N)$, where $\phi$ is Euler's totient function). This set in motion a small but steady stream of follow-up works at the interface of algebra, number theory and spectral theory; see e.g. \cite{BL} and references therein. The matrix $\{(n,m)\}_{n,m=1}^N$ has become known as the GCD matrix and $\{[n,m]\}_{n,m=1}^N$ as the LCM matrix, while $\{(n,m)^{-\tau}\}_{n,m=1}^N$ and $\{[n,m]^{-\tau}\}_{n,m=1}^N$ for $\tau\in\bbR$ have become known as power GCD and power LCM matrices. Furthermore, one can also consider the weighted power GCD and power LCM versions
\[
\left\{\frac{n^\sigma m^\sigma}{[n,m]^\tau}\right\}_{n,m=1}^N
\quad\text{ and }\quad
\left\{\frac{n^\sigma m^\sigma}{(n,m)^\tau}\right\}_{n,m=1}^N
\]
for $\sigma\in\bbR$. 
Because of the relation
\[
(n,m)[n,m]=nm,
\]
the last two forms are equivalent, up to relabelling the parameters $\tau$ and $\sigma$:
\[
\frac{n^\sigma m^\sigma}{(n,m)^\tau}
=
\frac{n^{\sigma-\tau}m^{\sigma-\tau}}{[n,m]^{-\tau}}.
\]
While early work on these and related matrices was mostly algebraic, since 1990s some interest arose in their properties as $N\to\infty$: asymptotics of various norms, largest and smallest eigenvalues etc., see \cite{LS,HL,Hauk,HEL,MH}. Our own interest in this subject arose in connection with multiplicative Toeplitz matrices \cite{H3}, although we will not discuss this connection here.

We are interested in the infinite matrix
\[
E:=\left\{\frac{n^\sigma m^\sigma}{[n,m]^\tau}\right\}_{n,m=1}^\infty
\]
considered as an operator on the Hilbert space $\ell^2(\bbN)$, $\bbN=\{1,2,3,\dots\}$. Observe that the entries of this matrix are homogeneous 
\[
\frac{(kn)^\sigma (km)^\sigma}{[kn,km]^\tau}=k^{-\tau+2\sigma}\frac{n^\sigma m^\sigma}{[n,m]^\tau}, \quad k\in\bbN
\]
of degree $-\rho$, where 
$$
\rho:=\tau-2\sigma.
$$
 In \cite{H2,H}, the case $\tau=1$ was considered; it was proved that $E$ is bounded on $\ell^2(\bbN)$ for $\sigma<\frac12$ and is in the Hilbert-Schmidt class  if $\sigma<\frac14$. In \cite{HP1}, the general case $\tau>0$ was considered; it was proved that for $\rho>0$ and $\tau+\rho>1$, the operator $E$ is positive semi-definite, bounded and compact on $\ell^2(\bbN)$, and the asymptotics of eigenvalues was computed; these results are described in the next subsection.

In this paper, we are interested in the case $\tau<0$. This introduces a new twist in the story: while $E$ is positive semi-definite for $\tau>0$, it is sign-indefinite for $\tau<0$, which can be checked by observing that the determinant of the top left $2\times2$ submatrix of $E$ is negative. 

Our main results are as follows. 
Assuming $\tau<0$, we will show that $E$ is bounded if and only if $\tau-\sigma>1/2$. If $E$ is bounded, it is compact and has infinitely many both positive and negative eigenvalues. We establish the asymptotics of these eigenvalues and show that their distribution is ``asymptotically symmetric'' with respect to reflection $\lambda\mapsto-\lambda$. Finally, in Section~\ref{sec.e} we indicate a connection of the spectral analysis of $E$ to the theory of Beurling primes.

\subsection{Background: the case $\tau>0$}

To set the scene, here we briefly recall the results of \cite{HP1} pertaining to the case $\tau>0$. We start by making two remarks. First, applying $E$ to the first element of the standard basis in $\ell^2$, we obtain the sequence $\{n^{\sigma-\tau}\}_{n=1}^\infty$. So if we want $E$ to be bounded, we need to this sequence to be an element of $\ell^2$, i.e. we need to assume $\tau-\sigma>1/2$. Second, the sequence of diagonal elements of $E$ is 
$\{n^{2\sigma-\tau}\}_{n=1}^\infty=\{n^{-\rho}\}_{n=1}^\infty$ where as above $\rho=\tau-2\sigma$ is the homogeneity exponent. It follows that the necessary condition for the compactness of $E$ is $\rho>0$. 

Now we recall the main result of \cite{HP1}:
\begin{theorem}\label{thm0}
Assume that the exponents $\sigma$ and $\tau$ satisfy 
\[
\rho>0, \quad \tau-\sigma>1/2, \quad \tau>0.
\]
Then $E$ is compact, positive definite and has a trivial kernel. Let $\{\lambda_n(E)\}_{n=1}^\infty$ be the sequence of eigenvalues of $E$, enumerated in non-increasing order with multiplicities taken into account. Then for some positive constant $\varkappa$ (which depends on $\sigma$ and $\tau$) we have
\[
\lambda_n(E)=\frac{\varkappa}{n^{\rho}}+{o}(n^{-\rho}), \quad n\to\infty.
\]
\end{theorem}
The constant $\varkappa$ was computed in \cite{HP1} in two particular cases:
\begin{align*}
\varkappa&=1,\quad \text{ if } \rho=1,
\\
\varkappa&=\sqrt{\zeta(2+4\sigma)}/\zeta(1+2\sigma),\quad \text { if } \rho=1/2;
\end{align*}
here $\zeta$ is the Riemann Zeta function. 

\subsection{Main result}
We come to our main result for $\tau<0$. As already discussed, condition $\tau-\sigma>1/2$ is necessary for the boundedness of $E$, and so we consider the range of parameters
\begin{equation}
\tau<0, \qquad \tau-\sigma>1/2.
\label{a0}
\end{equation}
Note that as a consequence, $\rho>1$.
In order to set up further notation, we state a preliminary form of our main result. 
\begin{theorem}\label{thm.a00}
Assume \eqref{a0}; then $E$ is bounded, compact and self-adjoint on $\ell^2(\bbN_0)$ and has a trivial kernel.
It has infinitely many negative eigenvalues and infinitely many positive eigenvalues. 
\end{theorem}
We will label the positive eigenvalues of $E$ by $\{\lambda^+_n(E)\}_{n=1}^\infty$, listing them with multiplicities and in non-increasing order. Similarly, we label the positive eigenvalues of $-E$ by $\{\lambda^-_n(E)\}_{n=1}^\infty$, listing them with multiplicities and in non-increasing order. Thus, the eigenvalues of $E$ are $\{\lambda^+_n(E)\}_{n=1}^\infty$ and $\{-\lambda^-_n(E)\}_{n=1}^\infty$.

With these notational preliminaries out of the way, we are ready to state the main result of the paper.

\begin{theorem}\label{thm.a0}
Assume \eqref{a0}. Then for some positive constant $\varkappa$ (which depends on $\sigma$ and $\tau$) we have 
\begin{align*}
\lambda^+_n(E)=\frac{\varkappa}{n^{\rho+\tau}}(1+o(1)), \quad n\to\infty,
\\
\lambda^-_n(E)=\frac{\varkappa}{n^{\rho+\tau}}(1+o(1)), \quad n\to\infty.
\end{align*}
\end{theorem}

\textbf{Discussion:}
\begin{enumerate}
\item
Observe that the exponent $\rho+\tau$ here is different from the exponent $\rho$ in Theorem~\ref{thm0}. 
\item
The asymptotic constant $\varkappa$ is the same for positive and negative eigenvalues of $E$. In other words, the distribution of eigenvalues of $E$ is asymptotically symmetric. 
\item
We are not able to compute $\varkappa$ in closed form, although we have an infinite product representation for $\varkappa$ in \eqref{d5}. 
\item
In Section~\ref{sec.e} we discuss a slight improvement of the error term $o(1)$ in the asymptotics.
\end{enumerate}

\subsection{Notation}

We denote by $\jap{\cdot,\cdot}$ the inner product in a Hilbert space. We set $a\vee b=\max\{a,b\}$ and $a\wedge b=\min\{a,b\}$.
We shall write $\bbN_0=\bbN\cup \{0\}$.

Let us discuss notation for eigenvalues. Let $A$ be a compact self-adjoint operator on a Hilbert space, with some positive and some negative eigenvalues. It will be convenient to have two systems of labelling the non-zero eigenvalues of $A$. The first one is the system that we have already used in the theorem above: we label the positive and negative eigenvalues separately, listing them as $\lambda^+_n(A)$ and $-\lambda^-_n(A)$. The second system is to label all non-zero eigenvalues of $A$ by $\{\lambda_n(A)\}_{n=1}^\infty$, listing them with multiplicities taken into account and in any convenient order. 
Thus, the list $\{\abs{\lambda_n(A)}\}_{n=1}^\infty$ is just the re-ordered union of the lists $\{\lambda^+_k(A)\}_{k=1}^\infty$ and $\{\lambda^-_\ell(A)\}_{\ell=1}^\infty$. Finally, sometimes we will start the enumeration not from $n=1$ but from $n=0$; this will be made clear in each case. 

It will be convenient to denote $t=-\tau>0$. Observing that $\sigma=(\tau-\rho)/2=-(\rho+t)/2$, we will rewrite $E$ as
\begin{equation}
E=\left\{\frac{[n,m]^t}{(nm)^{(\rho+t)/2}}\right\}_{n,m=1}^\infty.
\label{a1}
\end{equation}
In terms of the parameters $t$ and $\rho$, our assumptions \eqref{a0} rewrite as
\begin{equation}
t>0, \quad \rho>t+1.
\label{a2}
\end{equation}

\section{The strategy of proof}

For a natural number $n$, let us write its factorisation as a product of powers of primes
$$
n=\prod_{p\text{ prime}}p^{k_p},
$$
where $k_2,k_3,k_5,k_7,\dots$ are non-negative integers labeled by prime numbers, and $k_p=0$ except for finitely many primes $p$. Writing similarly $m=\prod_p p^{j_p}$, we see that the entries of $E$ (represented as in \eqref{a1}) can be written as 
\[
\frac{[n,m]^t}{(nm)^{(\rho+t)/2}}
=
\prod_{p\text{ prime}} \frac{p^{t(j_p\vee k_p)}}{p^{(\rho+t)(j_p+k_p)/2}}.
\]
In other words, the matrix $E$ can be viewed, at least formally, as an infinite tensor product 
\begin{equation}
E=\bigotimes_{p\text{ prime}}E_{p}, \quad
E_{p}=\{p^{-\frac12(\rho+t) j}p^{t(j\vee k)}p^{-\frac12(\rho+t) k}\}_{j,k\in\bbN_0}.
\label{b2}
\end{equation}
Our first step is the spectral analysis of the matrix $E_p$. Here we can regard $p>1$ as an arbitrary real number (not necessarily a prime). It is an elementary exercise to check that under the assumptions \eqref{a2} (in fact, it suffices to require $0<t<\rho$)  the matrix $E_p$ is Hilbert-Schmidt and therefore compact. We will prove
\begin{lemma}\label{lma.b1}
Assume \eqref{a2}; then the matrix $E_p$ has a trivial kernel, exactly one positive eigenvalue, which we denote by  $\lambda_0^+(E_p)$, and infinitely many negative eigenvalues, which we denote by $\{-\lambda_k^-(E_p)\}_{k=0}^\infty$. As $p\to\infty$, we have 
\begin{align}
\lambda^+_0(E_p)&=1+\calO(p^{-\rho+t}), 
\label{b3}
\\
\lambda^-_0(E_p)&=p^{-\rho+t}(1+\calO(p^{-\rho+t})+\calO(p^{-t})), 
\label{b4}
\\
\lambda^-_k(E_p)&\leq p^{-\rho k}\lambda^-_0(E_p), \quad k\geq1.
\label{b5}
\end{align}
\end{lemma}
We note that in the right hand side of \eqref{b4}, sometimes $\calO(p^{-\rho+t})$ dominates and sometimes $\calO(p^{-t})$, depending on the range of the parameters $\rho$, $t$. It will be convenient to equivalently rewrite \eqref{b4} as
\begin{equation}
\lambda^-_0(E_p)=p^{-\rho+t}(1+\calO(p^{-\delta})), \quad \delta=(\rho-t)\wedge t.
\label{b4a}
\end{equation}

In order to state our next result, we switch to the enumeration of eigenvalues of $E_p$ as $\{\lambda_k(E_p)\}_{k=0}^\infty$, listing them as follows:
\begin{align}
\lambda_0(E_p)&=\lambda_0^+(E_p)=1+\calO(p^{-\rho+t})>0,
\label{b6}
\\
\lambda_k(E_p)&=-\lambda_{k-1}^-(E_p)<0, \quad k\geq 1.
\notag
\end{align}
In \cite[Theorem 4.1]{HP1} we prove that the eigenvalues of $E$ are given by the products of the eigenvalues of the $E_p$. The argument of \cite{HP1} works both for $t<0$ (which was the main focus of that paper) and for $t>0$. 
\begin{theorem}  \cite[Theorem 4.1]{HP1}\label{thm.b2} 
Assume \eqref{a0}; then the matrix $E$ is bounded on $\ell^2(\bbN)$ and has a pure point spectrum. 
There exists an enumeration $\{\lambda_n(E)\}_{n=1}^\infty$ of the eigenvalues of $E$ (here we do not insist on monotonicity!) such that
\begin{equation}
\lambda_n(E)=\prod_{p \text{ \rm  prime}} \lambda_{k_p}(E_p), \quad\text{ if }\quad n=\prod_{p \text{ \rm prime}}p^{k_p}.
\label{b7}
\end{equation}
\end{theorem}
The product formula \eqref{b7} formally follows from the tensor product factorisation \eqref{b2}, although in \cite{HP1} we have chosen to avoid the highly abstract language of infinite tensor products and give a concrete proof adapted to this particular case. 

In \eqref{b7} we have $k_p=0$ for all but finitely many primes, and therefore the convergence of the infinite product reduces to the convergence of the infinite product
$$
\prod_{p \text{ prime}} \lambda_{0}(E_p);
$$
the latter product converges by \eqref{b6} because $\rho-t>1$ (see \eqref{a2}).

From the product formula \eqref{b7} and from Lemma~\ref{lma.b1} it follows that $E$ has infinitely many positive and infinitely many negative eigenvalues. 
Substituting the asymptotics of Lemma~\ref{lma.b1} into the product formula \eqref{b7}, we will prove our main result. Our main tool is the  Tauberian theorem due to Wiener-Ikehara. 

In Section~\ref{sec.c} we prove Lemma~\ref{lma.b1}. In Section~\ref{sec.d} we put it together with the product formula and prove Theorem~\ref{thm.a0}. In Section~\ref{sec.e} we discuss the connection with Beurling primes. 

\section{Proof of Lemma~\ref{lma.b1}}\label{sec.c}

\subsection{Spectral perturbation theory}
In this section, we use some simple statements from spectral perturbation theory of self-adjoint operators that ultimately follow from the min-max principle. For clarity, we recall the necessary statements here in general form; for the details, see e.g. \cite[Chapter~9]{BirmanSolomyak}.

Below $A$ and $B$ are compact self-adjoint operators in a Hilbert space. We write $A\geq B$ if $\jap{Ax,x}\geq \jap{Bx,x}$ for all elements $x$ in our Hilbert space (here $\jap{x,y}$ is the inner product of $x$ and $y$). For $\lambda>0$, we denote by $N((\lambda,\infty);A)$ the total number of eigenvalues (counting multiplicities) of $A$ in the interval $(\lambda,\infty)$ and similarly let $N((-\infty,-\lambda);A)=N((\lambda,\infty);-A)$. 

\begin{proposition}\cite[Theorem 9.2.7]{BirmanSolomyak}
\label{prp.st1}
Let $A$ and $B$ be compact self-adjoint operators with $A\geq B$. Then for any $\lambda>0$, 
\begin{align*}
N((\lambda,\infty);A)&\geq N((\lambda,\infty);B),
\\
N((-\infty,-\lambda);A)&\leq N((-\infty,-\lambda);B).
\end{align*}
\end{proposition}
Informally speaking, this means that a positive perturbations shifts all eigenvalues to the right, i.e. in the positive direction.

\begin{proposition}\cite[Theorem~9.3.3]{BirmanSolomyak}
\label{prp.st2}
Let $A$ and $B$ be compact self-adjoint operators such that $A- B$ is a rank one operator. 
Then for any $\lambda>0$, 
\begin{align*}
\abs{N((\lambda,\infty);A)-N((\lambda,\infty);B)}\leq 1,
\\
\abs{N((-\infty,-\lambda);A)-N((-\infty,-\lambda);B)}\leq 1.
\end{align*}
\end{proposition}
If $A-B$ is a rank one operator, then necessarily $A\geq B$ or $B\geq A$. Then the two previous propositions, taken together, show that the eigenvalues of $A$ and $B$ interlace. 
This is also known as the interlacing theorem. 

The final result we need is

\begin{proposition}\cite[Lemma 9.4.3]{BirmanSolomyak}
\label{prp.st}
Let $A$ and $B$ be bounded self-adjoint operators in a Hilbert space. 
Assume that $A$ has exactly one simple eigenvalue in an interval $(\alpha,\beta)$ and no other eigenvalues in the interval $(\alpha-2\norm{B},\beta+2\norm{B})$. Then $A+B$ has exactly one simple eigenvalue (and no other eigenvalues) in the interval $(\alpha-\norm{B},\beta+\norm{B})$. 
\end{proposition}
Informally speaking, this proposition tells us that the perturbation $B$ can ``displace'' the eigenvalues of $A$ by no more than $\norm{B}$. The condition that $A$ has no other eigenvalues in $(\alpha-2\norm{B},\beta+2\norm{B})$ ensures that $B$ cannot ``bring'' any other eigenvalues from the outside into our interval.

\subsection{Exactly one eigenvalue of $E_p$ is positive}

Let $T=\{T_{jk}\}_{j,k=0}^\infty$ be the upper triangular matrix with entries 
$$
T_{jk}=
\begin{cases}p^{-\frac{t}{2} j}p^{-\frac12(\rho-t)k}, & k\geq j,\\ 0, & k<j.\end{cases}
$$
It is easy to see that for $\rho>0$ the matrix $T$ is Hilbert-Schmidt. 
Let us compute $T^*T$:
\begin{align*}
[T^*T]_{jk}&=\sum_{\ell=0}^\infty T_{\ell j}T_{\ell k}
=
\sum_{\ell=0}^{j\wedge k}p^{-\frac{t}{2}\ell}p^{-\frac12(\rho-t)j}p^{-\frac{t}{2}\ell}p^{-\frac12(\rho-t)k}
\\
&=
p^{-\frac12(\rho-t)(j+k)}\sum_{\ell=0}^{j\wedge k}p^{-t\ell}
=p^{-\frac12(\rho-t)(j+k)}\frac{p^{-t(j\wedge k)-t}-1}{p^{-t}-1}
\\
&=p^{-\frac12(\rho-t)(j+k)}\frac{p^t-p^{-t(j\wedge k)}}{p^t-1}.
\end{align*}
Using the fact that $j+k=j\wedge k+j\vee k$, we find that 
\begin{align*}
(p^t-1)[T^*T]_{jk}&= p^tp^{-\frac12(\rho-t)(j+k)} - p^{-\frac12(\rho-t)(j+k) - t(j+k-j\vee k)}
\\
&=p^t\psi_j\psi_k - p^{-\frac12(\rho+t)j}p^{t(j\vee k)}p^{-\frac12(\rho+t)k},
\end{align*}
where 
\begin{equation}
\psi_k=p^{-\frac12(\rho-t)k}, \quad k\geq0.
\label{d1}
\end{equation}
Denoting by $\psi\in \ell^2(\bbN_0)$ the element with these coordinates and comparing with \eqref{b2}, 
we find
$$
E_p=-(p^t-1)T^*T+p^t\jap{\cdot,\psi}\psi,
$$
where $\jap{\cdot,\psi}\psi$ denotes the rank one operator mapping $x$ to $\jap{x,\psi}\psi$.
Since the operator $-(p^t-1)T^*T$ has no positive eigenvalues, by Proposition~\ref{prp.st2} (applied to any interval $(\lambda,\infty)$)  we find that at most one eigenvalue of $E_p$ is positive. 

Finally, we observe that $\jap{E_p e_0,e_0}=1>0$, where $e_0=(1,0,0,\dots)$ is the first element of the standard basis in $\ell^2(\bbN_0)$. It follows that $E_p$ has at least one positive eigenvalue. We conclude that $E_p$ has exactly one positive eigenvalue.

\subsection{The matrix structure of $E_p$}

Let us consider the matrix structure of $E_p$. 
The entries of the first row of $E_p$ are
$$
[E_p]_{0k}=p^{-\frac12(\rho-t)k}=\psi_k, \quad k\geq0.
$$
By symmetry, these are also the entries of the first column.
Further, the principal sub-matrix of $E_p$, obtained by deleting the first row and the first column of $E_p$, is
$$
\{p^{-\frac12(\rho+t)(j+k)}p^{t(j\vee k)}\}_{j,k=1}^\infty
=
\{p^{-\rho} p^{-\frac12(\rho+t)(j+k)}p^{t(j\vee k)}\}_{j,k=0}^\infty,
$$
i.e. this is the $p^{-\rho}$-scaled version of $E_p$. 

Let us express these facts in terms of block-matrix notation. 
Separating out the first co-ordinate of elements of $\ell^2(\bbN_0)$, we can write $\ell^2(\bbN_0)=\bbC\oplus \ell^2(\bbN)$. Writing $E_p$ as a $2\times 2$ block matrix with respect to this decomposition, we obtain
\begin{equation}
E_p=\begin{pmatrix}
1 & p^{-\frac12(\rho-t)}\jap{\cdot, \widetilde\psi}
\\
p^{-\frac12(\rho-t)}\widetilde\psi & p^{-\rho}\widetilde E_p
\end{pmatrix} \quad \text{ in } \bbC\oplus \ell^2(\bbN),
\label{d3}
\end{equation}
where $\widetilde\psi$ is the same vector as in \eqref{d1}, but considered as an element of $\ell^2(\bbN)$, i.e.  $\widetilde\psi=\{p^{-\frac12(k-1)(\rho-t)}\}_{k=1}^\infty$, and $\widetilde E_p$ is identical to $E_p$ but considered as acting on $\ell^2(\bbN)$ instead of $\ell^2(\bbN_0)$. 
The ``self-similar'' structure \eqref{d3} plays a key role in what follows.

\subsection{$2\times 2$ perturbation theory: the proof of \eqref{b3} and \eqref{b4}}

Starting from the decomposition \eqref{d3}, let us write $E_p$ as a sum
\begin{equation}
E_p=A+B, \quad 
A=\begin{pmatrix}
1 & p^{-\frac12(\rho-t)}\jap{\cdot, \widetilde\psi}
\\
p^{-\frac12(\rho-t)}\widetilde\psi & 0
\end{pmatrix},
\quad
B=\begin{pmatrix}
0 & 0
\\
0 & p^{-\rho}\widetilde E_p
\end{pmatrix}.
\label{d4}
\end{equation}
Our aim is to apply Proposition~\ref{prp.st} to this decomposition. 
It is a simple exercise in linear algebra to compute the spectrum of $A$. We easily find that $A$ has the eigenvalues $\{\lambda_-,0,\lambda_+\}$, where $0$ has infinite multiplicity in the spectrum, each of $\lambda_\pm$ has multiplicity one, and $\lambda_\pm$ are the solutions to the quadratic equation
$$
\lambda^2-\lambda-p^{-\rho+t}\norm{\widetilde\psi}^2=0.
$$
Solving the quadratic equation and using the fact that
$$
\norm{\widetilde\psi}^2=\norm{\psi}^2=1+\calO(p^{-\rho+t}),\quad p\to\infty,
$$
we find 
\begin{align*}
\lambda_+&=1+\calO(p^{-\rho+t}), 
\\
\lambda_-&=-p^{-\rho+t}(1+\calO(p^{-\rho+t})), 
\end{align*}
as $p\to\infty$. 

In order to use Proposition~\ref{prp.st}, we need to estimate the norm of $B$. We have
$$
\norm{B}=p^{-\rho}\norm{\widetilde E_p}=p^{-\rho}\norm{E_p}.
$$
We can use, for example, the Hilbert-Schmidt estimate for the norm of $E_p$:
$$
\norm{E_p}^2\leq \sum_{j,k=0}^\infty \abs{[E_p]_{jk}}^2
$$
which, after a short computation, yields $\norm{E_p}=\calO(1)$ as $p\to\infty$ and therefore 
$$
\norm{B}=\calO(p^{-\rho}), \quad p\to\infty.
$$
Applying Proposition~\ref{prp.st} to $A$, $B$ as in \eqref{d4}, we find that there is exactly one eigenvalue of $E_p$ in the $\calO(p^{-\rho})$-neighbourhood of $\lambda_+$ and also exactly one eigenvalue in the $\calO(p^{-\rho})$-neighbourhood of $\lambda_+$.

For clarity, let us record the same statement more formally. 
We apply Proposition~\ref{prp.st} to the interval 
$$
(\alpha,\beta)=(\lambda_+-p^{-\rho}\norm{E_p},\lambda_++p^{-\rho}\norm{E_p}).
$$ The proposition shows that there is exactly one eigenvalue of $E_p$ in the interval $(\lambda_+-2p^{-\rho}\norm{E_p},\lambda_++2p^{-\rho}\norm{E_p})$. As we already know that $E_p$ has exactly one positive eigenvalue, we conclude that 
$$
\lambda^+_0(E_p)=\lambda_++\calO(p^{-\rho}), 
$$
which yields \eqref{b3}. 

Similarly, let us apply Proposition~\ref{prp.st} to the interval 
$$
(\alpha,\beta)=(\lambda_--p^{-\rho}\norm{E_p},\lambda_-+p^{-\rho}\norm{E_p}).
$$ 
The proposition shows that there is exactly one eigenvalue of $E_p$ in the interval $(\lambda_--2p^{-\rho}\norm{E_p},\lambda_-+2p^{-\rho}\norm{E_p})$. Applying the proposition to $(-R,-2p^{-\rho}\norm{E_p})$, where $R>0$ is arbitrarily large, we find that there are no other eigenvalues of $E_p$ in $(-\infty,-p^{-\rho}\norm{E_p})$. We conclude that the eigenvalue that we have just located is indeed the smallest eigenvalue of $E_p$ and that it satisfies
$$
-\lambda^-_0(E_p)=\lambda_-+\calO(p^{-\rho}),
$$
which yields formula \eqref{b4}.

\subsection{Interlacing theorem: the proof of \eqref{b5}}

Essentially, \eqref{b5} is the consequence of Cauchy's interlacing theorem. Since the standard version of this theorem applies to finite matrices and we are working with operators in infinite dimensional space, we spell out the details of this argument. Write the decomposition \eqref{d3} as
$$
E_p=E_{\diag}+X, 
$$
where
$$
E_{\diag}=\begin{pmatrix} 1 & 0 \\ 0 & p^{-\rho} \widetilde E_p\end{pmatrix},
\quad
X=\begin{pmatrix}
0 & p^{-\frac12(\rho-t)}\jap{\cdot,\widetilde\psi} 
\\
p^{-\frac12(\rho-t)}\widetilde\psi & 0
\end{pmatrix}. 
$$
The matrix $X$ has rank two, with one positive and one negative eigenvalue; let us record this as
$$
X=\jap{\cdot,\varphi_+}\varphi_+ -\jap{\cdot,\varphi_-}\varphi_-.
$$
It follows that 
$$
E_p\geq E_{\diag}-\jap{\cdot,\varphi_-}\varphi_-.
$$
For any $\lambda>0$ by Proposition~\ref{prp.st1} we find
$$
N((-\infty,-\lambda);E_p)\leq N\bigl((-\infty,-\lambda); E_{\diag}-\jap{\cdot,\varphi_-}\varphi_-\bigr)
$$
and by by Proposition~\ref{prp.st2}
\begin{align*}
N\bigl((-\infty,-\lambda); E_{\diag}-\jap{\cdot,\varphi_-}\varphi_-\bigr)&\leq N((-\infty,-\lambda); E_{\diag})+1
\\
&= N((-\infty,-\lambda); p^{-\rho}\widetilde E_p)+1
\\
&= N((-\infty,-\lambda); p^{-\rho} E_p)+1.
\end{align*}
Putting this together, we find
$$
N((-\infty,-\lambda);E_p)\leq N((-\infty,-\lambda); p^{-\rho} E_p)+1;
$$
rewriting this in terms of the negative eigenvalues of $E_p$ gives
$$
\lambda_{k+1}^-(E_p)\leq 
\lambda_k^-(p^{-\rho}E_p)=
p^{-\rho}\lambda_k^-(E_p), \quad k\geq0.
$$
Iterating this, we arrive at the required estimate \eqref{b5}.

\subsection{$E_p$ has a trivial kernel}
Suppose $E_px=0$ for some $x\in \ell^2(\bbN_0)$. Let us use the definition \eqref{b2} of $E_p$. Relabelling $y_k=p^{-\frac12(\rho+t)k} x_k$, we find
\begin{equation}
\sum_{j=0}^\infty p^{t(j\vee k)}y_j=0, \quad k=0,1,2,\dots
\label{d6}
\end{equation}
Subtracting \eqref{d6} with $k=1$ from \eqref{d6} with $k=0$, we find $y_0=0$. 
Next, subtracting \eqref{d6} with $k=2$ from \eqref{d6} with $k=1$, we find $y_1=0$. 
Proceeding in the same way, by induction we find that $y=0$. 
The proof of Lemma~\ref{lma.b1} is complete. 

\section{Proof of Theorem~\ref{thm.a0}}\label{sec.d}

\subsection{The set-up}

We introduce the eigenvalue counting functions for $E$ as follows:
\begin{align*}
\mu^+(x)=\#\{n\in\bbN: \lambda^+_n(E)>1/x\}, 
\\
\mu^-(x)=\#\{n\in\bbN: \lambda^-_n(E)>1/x\}, 
\end{align*}
and $\mu(x)=\mu_+(x)+\mu_-(x)$, where $x>0$.
We will prove 
\begin{equation}
\mu^\pm(x)=\varkappa^{1/(\rho-t)}x^{1/(\rho-t)}+o(x^{1/(\rho-t)}), \quad x\to\infty, 
\label{e1}
\end{equation}
for both signs $+$ and $-$; 
this is equivalent to the desired asymptotics of Theorem~\ref{thm.a0}.

Consider the functions
\begin{equation}
f^\pm(s)=\int_0^\infty x^{-s}d\mu^\pm(x)=\sum_{n=1}^\infty \lambda^\pm_n(E)^s.
\label{e4}
\end{equation}
We will show that these integrals converge absolutely for $\Re s>\frac1{\rho-t}$. Furthermore, we will use the following Tauberian theorem of Wiener and Ikehara, see e.g. \cite[Theorem 4.1]{Korevaar} (the version given below differs only by scaling). 

\begin{theorem} [Wiener-Ikehara]\label{thm.tauberian} 
Let $\mu=\mu(x)$ be a non-negative non-decreasing function on $\bbR$, vanishing for all $x\leq x_0$ with some $x_0>0$ and such that the Mellin-Stieltjes transform 
$$
g(s)=\int_{x_0}^\infty x^{-s}d\mu(x)
$$
exists (i.e. the integral converges absolutely) for $\Re s>s_0>0$. Suppose that for some $A>0$, the analytic function 
$$
g(s)-\frac{A}{s-s_0}, \quad \Re s>s_0
$$
has a continuous extension to the closed half-plane $\Re s\geq s_0$. Then 
$$
\mu(x)x^{-s_0}\to A/s_0, \quad x\to\infty.
$$
\end{theorem}

We will eventually apply this theorem both to $g=f_+$ and $g=f_-$. 
Our auxiliary device will be the pair of functions 
$$
f(s)=f^+(s)+f^-(s), \quad h(s)=f^+(s)-f^-(s).
$$
Let us look at these two functions separately. 

\subsection{Analytic properties of $f(s)$ and $h(s)$}

For our next lemma, we recall that the Riemann zeta function $\zeta(s)$ defined for $\Re s>1$ by $\zeta(s) = \sum_{n=1}^\infty n^{-s}$, has the Euler product 
$$ \zeta(s) = \prod_p \Bigl(1-\frac1{p^s}\Bigr)^{-1}$$
in this region, and furthermore has an analytic continuation to $\bbC\setminus\{1\}$ with a simple pole at $s=1$ with residue 1. Furthermore, $\zeta(s)\ne 0$ on the line $\Re s\geq 1$. (See for example, \cite{Titchmarsh}.)

\begin{lemma}\label{prp.an}
The integrals \eqref{e4} converge absolutely for $\Re s>\frac1{\rho-t}$. Furthermore, with $f$ and $h$ as defined above, we have 
\begin{equation}
f(s) = \zeta((\rho-t)s)\widetilde{f}(s)\qquad \mbox{ and }\qquad h(s) = \frac{\widetilde{h}(s)}{\zeta((\rho-t)s)},
\label{e2}
\end{equation}
where $\widetilde{f}(s)$ and $\widetilde{h}(s)$ are analytic for $\Re s>s_1$ with some $s_1<\frac1{\rho-t}$.
\end{lemma}

\begin{proof}
We start with $f(s)$. For $p$ prime and $\Re s>0$, we define 
\[
f_p(s)=\sum_{k=0}^\infty \abs{\lambda_k(E_p)}^s. 
\]
We recall that by Lemma~\ref{lma.b1}, we have the exponential convergence $\lambda_k(E_p)\to0$ as $k\to\infty$, and therefore the series above converge absolutely for all $\Re s>0$. 
Let us use Lemma~\ref{lma.b1} to estimate $f_p(s)$, with the simplified notation \eqref{b4a}. 
We have 
\begin{align*}
f_p(s)=&(1+\calO(p^{-\rho+t}))^s+p^{-s(\rho-t)}(1+\calO(p^{-\delta}))^s
\\
&+\sum_{k=2}^\infty \calO(p^{-s\rho k-s(\rho-t)})(1+\calO(p^{-\delta}))^s
\\
=&1+p^{-s(\rho-t)}+\calO_s(p^{-\rho+t})+\calO_s(p^{-\delta}p^{-s(\rho-t)})+\calO_s(p^{-2s\rho}p^{-s(\rho-t)}),
\end{align*}
where $\calO_s$ means that the constants in the estimates are uniform in $s$ over compact sets in the half-plane $\Re s>0$. 
We rewrite this as 
\begin{align*}
f_p(s)=&\frac1{1-p^{-s(\rho-t)}}\widetilde f_p(s), 
\\
\widetilde f_p(s)=&1+\calO_s(p^{-\rho+t})+\calO_s(p^{-2s(\rho-t)})+\calO_s(p^{-\delta}p^{-s(\rho-t)})
\\
&+\calO_s(p^{-2s\rho}p^{-s(\rho-t)}).
\end{align*}
Now let us discuss the convergence of the infinite product
$$
\prod_{p\text{ prime}}f_p(s)
=
\prod_{p\text{ prime}}(1-p^{-s(\rho-t)})^{-1} \prod_{p\text{ prime}} \widetilde f_p(s).
$$
By the Euler product representation for the Riemann Zeta function, the first product in the right hand side converges absolutely to $\zeta((\rho-t)s)$ for $\Re s>1/(\rho-t)$. 
By the above estimates for $\widetilde f_p$, the second product in the right hand side converges absolutely in the half-plane where 
$$
2(\rho-t)\Re s>1, \quad
\delta+(\rho-t)\Re s>1,\quad
2\rho\Re s+(\rho-t)\Re s>1.
$$
This is the half-plane $\Re s>s_1$, where
$$
s_1 = \max\Bigl\{ \frac1{2(\rho-t)}, \frac{1-\delta}{\rho-t},\frac1{3\rho-t}\Bigr\}<\frac1{\rho-t}.
$$
Denoting 
$$
\widetilde f(s)=\prod_{p\text{ prime}} \widetilde f_p(s), \quad \Re s>s_1,
$$
we find that 
$$
\prod_{p\text{ prime}}f_p(s)=\zeta((\rho-t)s)\widetilde f(s),
$$
where the infinite product in the left hand side converges absolutely for $\Re s>1/(\rho-t)$ and $\widetilde f$ is analytic in the half-plane $\Re s>s_1$.

Now let $\{\lambda_n(E)\}_{n=1}^\infty$ be the enumeration of the eigenvalues of $E$ (both positive and negative) from Theorem~\ref{thm.b2}. Then from the eigenvalue product formula of Theorem~\ref{thm.b2} we find 
\[
f(s)=\sum_{n=1}^\infty \abs{\lambda_n(E)}^s=\prod_{p\text{ prime}}f_p(s), 
\]
where by the above analysis the series converges absolutely in the half-plane $\Re s>1/(\rho-t)$. We have proved the required statements for $f(s)$.

Now consider $h(s)$. This time we have
$$
h(s)=
\sum_{n=1}^\infty \sign(\lambda_n(E))\abs{\lambda_n(E)}^s, 
$$
where by the first part of the proof the series converges absolutely for $\Re s>1/(\rho-s)$. 
From the product formula of Theorem~\ref{thm.b2}, we find
$$
\sign(\lambda_n(E))\abs{\lambda_n(E)}^s
=\prod_{p \text{ \rm  prime}} 
\sign(\lambda_{k_p}(E_p))
\abs{\lambda_{k_p}(E_p)}^s,
$$
and therefore 
$$
h(s)=\prod_{p\text{ prime}}h_p(s), 
\quad
h_p(s)=\sum_{k=0}^\infty 
\sign(\lambda_{k}(E_p))
\abs{\lambda_{k}(E_p)}^s.
$$
Now we consider $h_p(s)$ in a very similar manner to the consideration of $f_p(s)$ above; the only difference will be a sign change in front of the $p^{-s(\rho-t)}$ term. 
We have 
\begin{align*}
h_p(s)=&(1+\calO(p^{-\rho+t}))^s-p^{-s(\rho-t)}(1+\calO(p^{-\delta}))^s
\\
&+\sum_{k=2}^\infty \calO(p^{-s\rho k-s(\rho-t)})(1+\calO(p^{-\delta}))^s
\\
=&1-p^{-s(\rho-t)}+\calO_s(p^{-\rho+t})+\calO_s(p^{-\delta}p^{-s(\rho-t)})+\calO_s(p^{-2s\rho}p^{-s(\rho-t)}),
\end{align*}
and therefore 
\begin{align*}
h_p(s)=&(1-p^{-s(\rho-t)})\widetilde h_p(s), 
\\
\widetilde h_p(s)=&1+\calO_s(p^{-\rho+t})+\calO_s(p^{-2s(\rho-t)})+\calO_s(p^{-\delta}p^{-s(\rho-t)})
\\
&+\calO_s(p^{-2s\rho}p^{-s(\rho-t)}).
\end{align*}
It follows that for $\Re s>1/(\rho-t)$ we have 
$$
h(s)=\prod_{p\text{ prime}}(1-p^{-s(\rho-t)}) \prod_{p\text{ prime}} \widetilde h_p(s)=\frac{\widetilde h(s)}{\zeta((\rho-t) s)}, \quad 
\widetilde h(s)=\prod_{p\text{ prime}} \widetilde h_p(s).
$$
As in the case for $\widetilde f(s)$, we see that $\widetilde h(s)$ is analytic in the half-plane $\Re s>s_1$. 
\end{proof}

From this lemma it is easy to deduce the required analyticity properties of $f$ and $h$. Recall that the function
\[
\zeta(s)-\frac1{s-1}
\]
is entire. Substituting this into the representation \eqref{e2} for $f(s)$, we find that 
\[
f(s)-\frac{\widetilde f(1/(\rho-t))}{(\rho-t)s-1}
\]
is analytic in the half-plane $\Re s>s_1$. 

For  $h(s)$, the argument is slightly more subtle. 
The fact that $\zeta(s)\ne 0$ for $\Re s\ge 1$ implies that $h(s)$ has an analytic continuation to the closed half-plane $\Re s\geq1/(\rho-t)$.
In contrast with the case of $f(s)$, there is no pole at $s=1/(\rho-t)$. 
Note that if it was known that $\zeta(s)\ne 0$ in a strip to the left of the 1-line, we could further deduce that $h(s)$ has an analytic continuation to a larger half-plane.  

\subsection{Putting this together}
Now we come back to the functions
$$
f^+(s)=\frac12(f(s)+h(s)) \quad\text{ and }\quad f^-(s)=\frac12(f(s)-h(s)).
$$
From the above analysis we find that, for both signs $+$ and $-$, 
$$
f^\pm(s)-\frac12\frac{\widetilde f(1/(\rho-t))}{(\rho-t)s-1}
$$
is analytic in the closed half-plane $\Re s\ge\frac1{\rho-t}$. In particular, $f^\pm$ have the same residues at $s=\frac1{\rho-t}$. 

 Applying the Wiener-Ikehara theorem, we find that 
\[
\mu^\pm(x) x^{-1/(\rho-t)}\to \frac12\widetilde f(1/(\rho-t)) \text{ as }x\to\infty.
\]
This proves \eqref{e1} and therefore Theorem~\ref{thm.a0} with 
$$
\varkappa=\Bigl(\frac12\widetilde f(1/(\rho-t))\Bigr)^{\rho-t}.
$$
Note that 
$$
\widetilde f(1/(\rho-t)) = \prod_{p\text{ prime}} \widetilde f_p(1/(\rho-t)) = \prod_{p\text{ prime}}\Bigl(1-\frac{1}{p}\Bigr)\sum_{k=0}^\infty \abs{\lambda_k(E_p)}^{1/(\rho-t)} > 0,
$$
as each term in the product is positive. Finally, for the exponent $\varkappa$ we find
\begin{equation}
\varkappa=2^{-\rho+t} \prod_{p\text{ prime}}\Bigl(1-\frac{1}{p}\Bigr)^{\rho-t}\sum_{k=0}^\infty \abs{\lambda_k(E_p)}.
\label{d5}
\end{equation}

\section{Connection with Beurling prime systems}\label{sec.e}

Here, without going into technical details, we discuss the connection to Beurling (or generalised) prime systems. While this connection may be interesting in its own right, it also allows one to improve the error terms in the eigenvalue asymptotics of Theorem~\ref{thm.a0}.

Briefly, a Beurling prime system $\mathcal{P}$ consists of a sequence of real numbers $1 < r_1 \le\cdots \le r_n \to\infty$ (called Beurling primes) and all possible products of powers of these (called Beurling integers), say, $1 = n_1 < n_2 \le \cdots$. One associates with $\mathcal{P}$ the Beurling zeta function 
$$ \zeta_\mathcal{P}(s)  = \prod_{k=1}^\infty \frac1{1-r_k^{-s}} = \sum_{k=1}^\infty \frac1{n_k^s}.$$
See, e.g., \cite{DZ} for the details. 

First, let us briefly recall what happens in the case $t<0$ from \cite{HP1}. 
Let $\{\lambda_n(E)\}_{n=1}^\infty$ be the enumeration of the eigenvalues of $E$ (both positive and negative) from Theorem~\ref{thm.b2}. 
Define
$$ \gamma_n = \frac{\lambda_n(E)}{\lambda_1(E)} \quad (n\in\bbN) \quad \mbox{ and }  \quad \gamma_{k,p} = \frac{\lambda_k(E_p)}{\lambda_0(E_p)} \quad (k\in\bbN_0)$$
and view $\mathcal{P}=\{ \gamma_{1,p}^{-\frac1{\rho}}\}_{p \,{\rm prime}}$ as a sequence of Beurling primes. It was shown that 
\begin{equation}
\gamma_{1,p}^{-\frac1{\rho}} = p(1 + \calO(p^{-\delta}))
\label{f1}
\end{equation}
for some $\delta>0$ and $f(s) = \zeta_\mathcal{P}(\rho s)\widetilde{g}(s)$, where $\widetilde{g}$ is analytic and bounded in some half-plane containing $\frac1{\rho}$. Thus the asymptotics of the eigenvalues $\lambda_n(E)$ is essentially given by the asymptotics of the Beurling integers associated to $\mathcal{P}$. By a result of Balazard \cite[Theorem 9]{Balazard}, the counting function of the corresponding Beurling integers satisfies
\begin{equation}
 N(x):=\sum_{n_k \le x}1 = c_1x(1 + \calO(e^{-c_2\sqrt{\log x \log \log x}}))
\label{e3}
\end{equation}
for some $c_1,c_2 > 0$. Rescaling gives $\lambda_n(E) = \frac{\varkappa}{n^\rho}(1+\calO(e^{-c_3\sqrt{\log n\log\log n}}))$ for some $c_3>0$. Furthermore, we remark here that obtaining an error bound of order $\calO(n^{-\eta})$ for some $\eta>0$ is equivalent to having $\zeta_\mathcal{P}$ being of finite order in some strip to the left of $\Re s=1$. 
In general this is not expected to follow simply from the fact that the Beurling primes satisfy \eqref{f1}. For example, even in the case of  the system $\mathcal{P}_0$ whose Beurling primes are $p-1$ (for $p$ an odd prime), the error $\calO(x^{-\eta})$ is not expected to hold in \eqref{e3} (see \cite{Pom}). To explain briefly: with $\phi(n)$ denoting Euler's totient function, one has (by multiplicativity)
$$ \sum_{n=1}^\infty \frac1{\phi(n)^s} = \prod_p \sum_{k=0}^\infty \frac1{\phi(p^k)^s} = \prod_p\biggl(1+\frac1{(p-1)^s}+\cdots\Biggr) = \zeta_{\mathcal{P}_0}(s)H(s),$$
where $H(s)$ is holomorphic and bounded for $\Re s\ge \frac12+\delta$ for every $\delta>0$. If $\zeta_{\mathcal{P}_0}$ has finite order in a strip to the left or $\Re s=1$, then it would follow that 
$$ \sum_{\phi(n)\le x} 1 = cx + \calO(x^\alpha)$$
for some $c>0$ and $\alpha<1$. But the error in the above is conjectured to be $\Omega(x\exp\{-(1+\epsilon)\log x\log\log\log x/\log\log x\})$ for every $\epsilon>0$ (see \cite{Pom}, page 85 near the end of \S1) and hence certainly is $\Omega(x^{1-\epsilon})$. In other words, we need something more than just \eqref{f1} to have an error $\calO(n^{-\eta})$.

We note that \eqref{f1} implies that $\zeta_\mathcal{P}$ and $\zeta$ have the same zeros in the half plane where $\Re s>1-\delta$ (see \cite{GS}, Theorem 1 and the subsequent remark).

For the present case ($t>0$), we take $\mathcal{P}=\{ |\gamma_{1,p}|^{-1/(\rho-t)}\}_{p \,{\rm prime}}$ as a sequence of Beurling primes. We have
$$ \bigl|\gamma_{1,p}\bigr|^{-\frac1{\rho-t}} = p(1 + \calO(p^{-\delta}))$$
for some $\delta>0$. Following the same procedure, we find that
$$ f(s) = \zeta_\mathcal{P}((\rho-t)s)f_1(s) \quad\mbox{ and }\quad h(s) = \frac{h_1(s)}{\zeta_\mathcal{P}((\rho-t)s)},$$
where $f_1$ and $h_1$ are analytic and bounded in some half-plane containing $\frac1{\rho-t}$. Indeed, we start by observing that $\lambda_0(E_p)>0$ for all $p$ and as a consequence $\lambda_1(E)>0$. 
Further, we find
\begin{align*}
f(s) &= \lambda_1(E)^s\sum_{n=1}^\infty |\gamma_n|^s = \lambda_1(E)^s\prod_{p\text{ prime}} \biggl(1+\sum_{k=1}^\infty |\gamma_{k,p}|^s\biggr)
\\
& = \lambda_1(E)^s\prod_{p\text{ prime}} \frac1{1-|\gamma_{1,p}|^s}\prod_{p\text{ prime}} (1-|\gamma_{1,p}|^s)\biggl(1+\sum_{k=1}^\infty |\gamma_{k,p}|^s\biggr) 
\\
&= \zeta_\mathcal{P}((\rho-t)s)f_1(s),
\end{align*}
and similarly
\begin{align*}
h(s) &= \lambda_1(E)^s\sum_{n=1}^\infty \sign(\gamma_n)|\gamma_n|^s 
= 
\lambda_1(E)^s\prod_{p\text{ prime}} \biggl(1-\sum_{k=1}^\infty|\gamma_{k,p}|^s\biggr)
\\
&= \lambda_1(E)^s\prod_{p\text{ prime}}  (1-|\gamma_{1,p}|^s)\prod_{p\text{ prime}}\frac1{1-|\gamma_{1,p}|^s}\biggl(1-\sum_{k=1}^\infty |\gamma_{k,p}|^s\biggr) 
\\
&= \frac{h_1(s)}{\zeta_\mathcal{P}((\rho-t)s)}.
\end{align*}
As a result, by the application of  \cite[Theorem 9]{Balazard}, the asymptotic relation \eqref{e3} again holds, which implies
$$ \mu^+(x)+\mu^-(x) = cx^{\frac{1}{\rho-t}}(1+\calO(e^{-c_2\sqrt{\log x \log \log x}})).$$
For $h$, we require the knowledge of
$$ M(x):= \sum_{n_k\le x} \mu_{\mathcal{P}}(n_k),$$
where $\mu_\mathcal{P}(n_k)=(-1)^N$ if $n_k$ is a product of $N$ distinct Beurling primes and zero otherwise, generalising the M\"{o}bius function. By the same methods as in \cite{Balazard}, one can show that  $M(x)=\calO(xe^{-c\sqrt{\log x \log \log x}})$ for some $c>0$. Thus, also 
\[
\mu^+(x)-\mu^-(x) = \calO(x^{\frac{1}{\rho-t}}e^{-c_2\sqrt{\log x \log \log x}})
\] 
and the error terms in Theorem~\ref{thm.a0} are of the form $\calO(e^{-c\sqrt{\log n \log \log n}})$. 

To have an error of the form $\calO(n^{-\eta})$ however, requires that $\zeta_\mathcal{P}(s)\ne 0$ in a strip to the left of $\Re s=1$ (or somehow  all such zeros are cancelled by zeros of $h_1$). As we noted that $\zeta_\mathcal{P}$ and $\zeta$ have the same zeros in a strip to the left of $\Re s=1$, this is equivalent to $\zeta(s)\ne 0$ in such a strip; i.e. a weak version of the Riemann Hypothesis.


\end{document}